\documentclass{amsproc} 
\usepackage{amssymb}

\newcommand{\e}{\varepsilon}
\renewcommand{\k}{\varkappa}
\newcommand{\f}{{\varphi}}

\DeclareMathOperator{\iindex}{index}
\DeclareMathOperator{\sign}{sign}
\DeclareMathOperator{\dist}{dist}
\DeclareMathOperator{\supp}{supp}

\DeclareMathOperator{\Ran}{Ran}
\DeclareMathOperator{\Dom}{Dom}
\DeclareMathOperator{\Ker}{Ker}
\DeclareMathOperator{\Tr}{Tr}
\DeclareMathOperator{\tr}{tr}

\renewcommand\Im{\hbox{{\rm Im}}\,}

\newcommand{\abs}[1]{\lvert#1\rvert}
\newcommand{\norm}[1]{\lVert#1\rVert}
\newcommand{\Norm}[1]{\left\lVert#1\right\rVert}

\newcommand{\wt}{\widetilde}

\newcommand{\R}{{\mathbb R}}
\newcommand{\Z}{{\mathbb Z}}
\newcommand{\C}{{\mathbb C}}
\renewcommand{\H}{{\mathcal H}}

\numberwithin{equation}{section}

\newtheorem{theorem}{Theorem}[section]
\newtheorem{lemma}[theorem]{Lemma}
\newtheorem{corollary}[theorem]{Corollary}
\newtheorem{proposition}[theorem]{Proposition}

\theoremstyle{definition}

\theoremstyle{remark}
\newtheorem{remark}[theorem]{Remark}

\numberwithin{equation}{section}

\begin{document}

\title[Spectral flow, Fredholm index, and spectral shift]{The spectral flow, the Fredholm index, and the spectral shift function}

\author{Alexander Pushnitski}

\address{Department of Mathematics, King's College, London. Strand, London, WC2R 2LS, U.K.}
\email{alexander.pushnitski@kcl.ac.uk}


\dedicatory{Dedicated to M.~Sh.~Birman on the occasion of his 80th birthday}

\begin{abstract}
We discuss the well known ``Fredholm index=spectral flow'' theorem and show that 
it can be interpreted as a limit case of an identity involving two spectral shift functions.
\end{abstract}

\maketitle

\section{Introduction}

\subsection{Background}
Let $A(t)$, $t\in\R$, be a family of self-adjoint operators in a separable Hilbert 
space $\H$ such that the limits 
\begin{equation}
A^\pm=\lim_{t\to\pm\infty} A(t)
\label{a1}
\end{equation}
exist in an appropriate sense. 
In the Hilbert space $L^2(\R;\H)$, consider the operator 
\begin{equation}
D_A=\frac{d}{dt}+A(t),
\text{ i.e. } (D_A u)(t)=\frac{du(t)}{dt}+A(t)u(t).
\label{a2}
\end{equation}
It is well known  (see e.g. \cite{Callias,RobbinS} and references to earlier work therein)
that, under the appropriate assumptions on $A(t)$, the Fredholm index of the operator $D_A$ 
equals the spectral flow of the family $\{A(t)\}_{t\in\R}$
through zero. The spectral flow through zero should be understood as the number of 
eigenvalues of $A(t)$ (counting multiplicities) that cross zero from left to right minus 
the number of eigenvalues
of $A(t)$ that cross zero from right to left as $t$ grows from $-\infty$ to $+\infty$.
The ``Fredholm index = spectral flow through zero'' theorem is one of the large family of 
index theorems; see e.g. \cite{Callias} for discussion.

The ``Fredholm index = spectral flow through zero'' theorem is usually considered  
under the assumption that the spectra of the operators $A(t)$ are discrete, 
at least on some interval containing zero.
The purpose of this note is to show that this assumption can be lifted at the expense of the trace class
assumption
\begin{equation}
\int_{-\infty}^\infty \Norm{A'(t)}_{S_1} dt<\infty,
\quad A'(t)\equiv  \frac{d A(t)}{dt},
\label{a3}
\end{equation}
where $\norm{\cdot}_{S_1}$ is the trace norm in $\H$, $\norm{A}_{S_1}=\sqrt{\tr(A^*A)}$.
Assumption \eqref{a3} ensures that the difference $A^+-A^-$ is a trace class operator,
which allows one to use the notion of  \emph{M.~G.~Krein's spectral shift function}
for the pair $A^+$, $A^-$. 
The point we would like to make is  that the ``Fredholm index = spectral flow'' 
theorem can be understood as a particular limiting case of a fairly general  identity (see \eqref{a11} below) 
involving two spectral shift functions.
This identity might be interesting in its own right.

It is a pleasure to dedicate this note to M.~Sh.~Birman, 
who has taught me (among many other useful things) 
to think of the  spectral shift function whenever two self-adjoint operators are involved.

\subsection{Notation} We denote by $S_1$ and $S_2$ the trace class and the Hilbert-Schmidt class, 
with the norms $\norm{\cdot}_{S_1}$ and $\norm{\cdot}_{S_2}$. 
For a self-adjoint operator $A$ and an interval $\delta\subset \R$, we denote by $E_A(\delta)$ 
the spectral projection of $A$ corresponding to $\delta$. We denote by $N_A(\delta)=\Tr E_A(\delta)$
the total number of eigenvalues (counting multiplicities) of $A$ in the interval $\delta$. 
For self-adjoint semi-bounded from below operators $A$ and $B$, the inequality $A\leq B$ is 
understood in the quadratic form sense, i.e. for all sufficiently large $a>0$, 
$\Dom((A+a)^{1/2})\supset \Dom((B+a)^{1/2})$ and for all $f\in  \Dom((B+a)^{1/2})$, 
$\norm{(A+a)^{1/2}f}\leq \norm{(B+a)^{1/2}f}$.

\subsection{The spectral shift function}
Here we recall the necessary facts from the spectral shift function theory. 
See the original paper \cite{Krein} or a survey \cite{BYa} 
or a book \cite{Yafaev} for the details.

Let $H$ and $\wt H$ be self-adjoint operators in a Hilbert space. The simplest situation in 
which the spectral shift function can be defined is when the difference
$\wt H-H$ belongs to the trace class $S_1$. Then there exists a unique function $\xi(\cdot; \wt H,H)\in L^1(\R)$ 
such that the Lifshits-Krein trace formula
\begin{equation}
\Tr(f(\wt H)-f(H))=\int_{-\infty}^\infty \xi(\lambda; \wt H,H)f'(\lambda)d\lambda
\label{a6}
\end{equation}
holds true for every $f\in C_0^\infty(\R)$. The function
$\xi(\cdot; \wt H,H)$ is called the spectral shift function for the pair $\wt H$, $H$. 
In fact, the class of admissible functions $f$  in \eqref{a6} is much wider than $C_0^\infty(\R)$; see 
\cite{BYa,Yafaev} for the details and references to the literature.

The assumption $\wt H-H\in S_1$ is very restrictive in applications. 
Suppose instead that $\wt H$ and $H$ are non-negative (in the quadratic form sense) 
self-adjoint operators such that 
\begin{equation}
(\wt H-z)^{-1}-(H-z)^{-1}\in S_1
\label{a12}
\end{equation}
for some (and hence for all) $z\in \C\setminus[0,\infty)$. Then there exists a unique function 
$\xi(\cdot;\wt H,H)\in L^1(\R,(1+\lambda^2)^{-1}d\lambda)$, $\supp \xi\in[0,\infty)$ such that 
the trace formula \eqref{a6} holds true for all $f\in C_0^\infty(\R)$.

Next, assuming either $\wt H-H\in S_1$  or \eqref{a12} holds true,
suppose that for some (possibly semi-infinite) open interval $\Delta\subset\R$ 
we have $\sigma_{ess}(H)\cap \Delta=\emptyset$. 
By Weyl's theorem on the invariance of 
the essential spectrum with respect to compact perturbations, 
we also have $\sigma_{ess}(\wt H) \cap\Delta=\emptyset$.   
Then, as it is not difficult to see from \eqref{a6}, for any  
$a,b\in\Delta\setminus(\sigma(H)\cup\sigma(\wt H))$, $a<b$, we have 
\begin{equation}
\xi(b;\wt H,H)-\xi(a,\wt H,H)=N_H(a,b)-N_{\wt H}(a,b),
\label{a7}
\end{equation}
where $N_H(a,b)$ is the number of eigenvalues (counting multiplicities) of $H$ in $(a,b)$. 
Formula \eqref{a7} remains true in the case $a=-\infty$ (or $b=\infty$).

\subsection{ The spectral shift function and the spectral flow}
Let $H(\alpha)$, $\alpha\in[0,1]$, be a family of self-adjoint operators such that
the operators $H(\alpha)-H(0)$ belong to the trace class and depend continuously on $\alpha$ 
in the trace norm. Then the function $\xi(\cdot; H(\alpha),H(0))$ is well defined and continuous in $\alpha$ 
as an element of $L^1(\R)$. 

As noted above, by Weyl's theorem $\sigma_{ess}(H(\alpha))$ is independent of $\alpha\in[0,1]$. 
Suppose that $\sigma_{ess}(H(\alpha))\cap \Delta=\emptyset$ for some interval $\Delta\subset\R$.
Then we claim that for any $\lambda\in\Delta\setminus(\sigma(H(1))\cup \sigma(H(0)))$, the spectral shift function
$\xi(\lambda; H(1),H(0))$ equals the spectral flow of the family $H(\alpha)$ through $\lambda$ as $\alpha$ grows
from $0$ to $1$: 
\begin{multline}
\xi(\lambda; H(1),H(0))
=
\langle \text{the number of eigenvalues of $H(\alpha)$ which cross $\lambda$ rightwards}\rangle
\\
-
\langle \text{the number of eigenvalues of $H(\alpha)$ which cross $\lambda$ leftwards}\rangle
\label{a8}
\end{multline}
as long as the r.h.s. is finite. 

In order to justify \eqref{a8}, first suppose that there exists $\lambda_0<\lambda$ such that 
\begin{equation}
\lambda_0\in\Delta\setminus \left(\cup_{\alpha\in[0,1]} \sigma(H(\alpha))\right).
\label{a9}
\end{equation}
Then it is not difficult to check that $\xi(\lambda_0;H(\alpha),H(0))=0$ for all $\alpha$
and so \eqref{a7} (with $a=\lambda_0$, $b=\lambda$) yields
\begin{equation}
\xi(\lambda; H(\alpha),H(0))=N_{H(0)}(\lambda_0,\lambda)-N_{H(\alpha)}(\lambda_0,\lambda).
\label{a10}
\end{equation}
Considering the r.h.s. of \eqref{a10} as a function of $\alpha$, it is easy to see that
$$
\langle \text{r.h.s. of \eqref{a10} with $\alpha=1$ }\rangle=\langle\text{r.h.s. of \eqref{a8}} \rangle
$$
whenever  the r.h.s. of \eqref{a8} is well defined. Thus, \eqref{a8} holds true. 

In general, $\lambda_0$ as in \eqref{a9} may not exist, but we can always split $[0,1]$ into sufficiently 
small subintervals $\delta_i$ such that for each family $\{H(\alpha)\mid \alpha\in \delta_i\}$, 
$\lambda_0$ can be chosen appropriately. Then formula \eqref{a8} can be obtained by adding up 
the formulas corresponding to all the subintervals.

\subsection{ Main result}
It will be convenient to write $A(t)=A^-+B(t)$, where $A^-$ is an arbitrary
self-adjoint operator in $\H$ and $B(t)$ is a family of trace class operators 
such that the derivative $B'(t)=\frac{dB(t)}{dt}$ exists in the trace norm and 
\begin{equation}
\int_{-\infty}^\infty \Norm{B'(t)}_{S_1} dt<\infty.
\label{a13}
\end{equation}
This assumption ensures that the limits $B^\pm=\lim_{t\to\pm\infty} B(t)$ 
exist in trace norm. We assume that $B^-=0$ (of course, this is merely
a normalization condition) and  define $A^+=A^-+B^+$. 
According to this definition, 
for all $t\in\R$ the operators $A(t)$ have the same domain $\Dom(A(t))=\Dom(A^-)$. 

Consider the operator $D_A$ (see \eqref{a2}) in the Hilbert space $L^2(\R,\H)$ with the
domain $\Dom(D_A)$ consisting of all 
$u$ from the Sobolev space $W^1_2(\R,\H)$ such that 
$$
u(t)\in \Dom(A^-) \text{ for all $t$ and }
\int_{-\infty}^\infty (\norm{u'(t)}^2+\norm{A^-u(t)}^2)dt<\infty.
$$
The operator $D_A$ is closed. This can be seen as follows.
Let us write $D_A$ as a sum $D_A=D_{A^-}+B$.
The operator $B$ of multiplication by $B(t)$ is bounded in $L^2(\R;\H)$,
so the question reduces to the closedness of $D_{A^-}$.
Note that $D_{A^-}=i(\frac1i \frac{d}{dt})+A^-$, and the 
self-adjoint operators $(\frac1i \frac{d}{dt})$ and $A^-$ in $L^2(\R;\H)$
commute. 
Using the spectral representations of
$(\frac1i \frac{d}{dt})$ and $A^-$, it is easy to see 
that $D_{A^-}$ is closed on the domain $\Dom(D_{A^-})=\Dom(D_{A})$.

The same argument shows that the adjoint operator $D_A^*$ is defined  as a closed operator 
on the same domain as $D_A$.
Below we consider the self-adjoint operators 
$$
H=D_A^* D_A\quad \text{ and }\quad \wt H=D_A D_A^*.
$$
As the difference $A^+-A^-$ is a trace class operator, we can consider
the spectral shift function $\xi(\lambda;A^+,A^-)$, $\lambda\in\R$.
The first part of the Theorem below shows that 
the spectral shift function $\xi(\lambda;\wt H,H)$ is also well defined. The following Theorem 
relates these two spectral shift functions. 
\begin{theorem}\label{thm.a1}
Assume \eqref{a3} and let $H$, $\wt H$ be as defined above.
For any $z\in\C\setminus[0,\infty)$, the difference
$(\wt H-z)^{-1}-(H-z)^{-1}$ belongs to the trace class. 
For a.e. $\lambda>0$, we have the identity 
\begin{equation}
\xi(\lambda;\wt H,H)=\frac1\pi\int_{-\sqrt{\lambda}}^{\sqrt{\lambda}}\xi(s;A^+,A^-)\frac{ds}{\sqrt{\lambda-s^2}},
\label{a11}
\end{equation}
where the integral in the r.h.s. converges absolutely.
\end{theorem}

\begin{corollary}\label{cr.a3}
Suppose that $0$ is not in the spectrum of $A^+$ or $A^-$.
Then the operator $D_A$ is Fredholm and
\begin{equation}
\iindex D_A
=
\dim \Ker D_A-\dim \Ker D_A^*
=
\xi(0;A^+,A^-).
\label{a5}
\end{equation}
\end{corollary}
Generally speaking, the spectral shift function $\xi(\lambda;A^+,A^-)$
is defined as an element of $L^1(\R)$, so it does not 
make sense to speak of its value at a fixed point $\lambda=0$.
However, the assumption of Corollary~\ref{cr.a3} 
implies that $\xi(\lambda;A^+,A^-)$ is constant near $\lambda=0$, 
and so $\xi(0;A^+,A^-)$ is well defined.

According to \eqref{a8}, the r.h.s. of \eqref{a5} coincides with the 
spectral flow of the family $A(t)$ through zero as long as the spectral flow is well defined. 
Thus, Corollary~\ref{cr.a3} can be interpreted as the generalisation of the  
``Fredholm index=spectral flow'' theorem.

Corollary~\ref{cr.a3},  under various sets of conditions on $A(t)$, is well known; see, e.g. 
\cite{RobbinS,Callias,BGS} and references to earlier works therein.

\subsection{ The strategy of the proof }
The proof of Theorem~\ref{thm.a1} is based on an identity due to \cite[(3.14)]{Callias} which we state
below as Proposition~\ref{p1}.  
In order to state this identity, let us fix the principal branch of the square root in $\C\setminus(-\infty,0]$. 
For any $z\in\C\setminus[0,\infty)$ and any $s\in\R$, we denote 
$$
g_z(s)=\frac{s}{\sqrt{s^2-z}}.
$$
The formula below involves traces of operators in $\H$
and in $L^2(\R;\H)$.
We denote by $\tr$ the trace in $\H$ and by $\Tr$ the trace in $L^2(\R;\H)$.
\begin{proposition}\label{p1}
Assume \eqref{a3} and let $\wt H$, $H$ be as defined above. Then for any 
 $z\in\C\setminus[0,\infty)$, the difference $g_z(A^+)-g_z(A^-)$ belongs to the 
 trace class in $\H$ and 
\begin{equation}
\Tr((\wt H-z)^{-1}-(H-z)^{-1})
=
\frac1{2z}\tr(g_z(A^+)-g_z(A^-)).
\label{a4}
\end{equation}
\end{proposition}
The identity \eqref{a4} was proven\footnote{This identity is stated in \cite{Callias} with a wrong sign; compare with \cite[Example 4.1]{BGS}} 
in \cite[(3.14)]{Callias} for the case $\dim \H<\infty$
as a particular case of a more general trace identity. We will give a more streamlined proof 
of \eqref{a4}, based on the ideas from \cite{BGS}, where \eqref{a4} was proven in the case
$\dim \H=1$.  This plan of the proof is as follows. 
We first note that the operators $H$ and $\wt H$ can be represented as
\begin{gather}
H=-\frac{d^2}{dt^2}+Q(t),
\qquad
\wt H=-\frac{d^2}{dt^2}+\wt Q(t),
\label{b1}
\\
\text{ where }\quad Q(t)=A(t)^2-A'(t), \qquad  \wt Q(t)=A(t)^2+A'(t).
\label{b1a}
\end{gather}

Then, following \cite{MOlmedilla}, we express the integral kernels of the resolvents $(H-z)^{-1}$ and $(\wt H-z)^{-1}$ 
in terms of the solutions to the operator differential equation 
$-F''+QF=zF$. Integrating the traces of these resolvent kernels over the diagonal, 
we obtain the expression the l.h.s. of \eqref{a4}.

Given \eqref{a4}, one can derive Theorem~\ref{thm.a1} fairly easily by applying the 
Lifshits-Krein trace formula \eqref{a6} to both sides of \eqref{a6}. 
This is done in 
Section~\ref{sec.b}. Proposition~\ref{p1} in the case $\dim\H<\infty$ and  Corollary~\ref{cr.a3} 
are also proven in Section~\ref{sec.b}. 

In Section~\ref{sec.c} (which is of a technical nature) we use an approximation argument to extend 
 Proposition~\ref{p1} 
from the case $\dim\H<\infty$ to the case $\dim \H=\infty$. 

\subsection{ Acknowledgments } This work originated from a question asked by Yuri Tomilov. 
I am  grateful to Yuri Tomilov and Yuri Latushkin for many useful discussions 
and encouragement, and to Nikolai Filonov for providing the proof
to Lemma~\ref{lma.c1}. I am also grateful to M.~Sh.~Birman and to F.~Gesztesy for discussing
the results of the paper.

\section{Proofs}\label{sec.b}

\subsection{Trace class inclusions}
Our first task is to show that the operators in both sides of \eqref{a4} belong to the trace class. 

\begin{lemma}\label{l.b1}
Under the assumptions of Theorem~\ref{thm.a1}, the operators $(H-z)^{-1}-(\wt H-z)^{-1}$ 
and $g_z(A^+)-g_z(A^-)$ belong to the trace class for any $z\in\C\setminus[0,\infty)$.
\end{lemma}

Note that Lemma~\ref{l.b1} in particular ensures that the spectral shift function $\xi(\cdot; \wt H,H)$ 
is well defined. 

Let us introduce some notation. 
We denote by $H_0$ the self-adjoint operator $-d^2/dt^2$ in $L^2(\R,\H)$ and by 
$R_0(z)=(H_0-z)^{-1}$ its resolvent. We also denote $R(z)=(H-z)^{-1}$ and $\wt R(z)=(\wt H-z)^{-1}$.

First, we need

\begin{lemma}\label{lma.b5}
Let $V(t)$, $t\in\R$, be a family of trace class operators in $\H$ such that 
\begin{equation}
\int_{-\infty}^\infty \norm{V(t)}_{S_1}dt <\infty,
\label{b27}
\end{equation}
and let $V$ be the operator in $L^2(\R,\H)$, $(Vu)(t)=V(t)u(t)$.
Then for any $z<0$, one has
\begin{equation}
R_0(z)^{1/2}VR_0(z)^{1/2}\in S_1
\text{ and }
\norm{ R_0(z)^{1/2}VR_0(z)^{1/2} }_{S_1}\leq \frac{1}{4\sqrt{\abs{z}}}\int_{-\infty}^\infty \norm{V(t)}_{S_1}dt. 
\label{b28}
\end{equation}
In particular, $V$ is a relatively form compact perturbation of $H_0$. 
\end{lemma}
\begin{proof}
Write $V(t)=\abs{V(t)}^{1/2} \sign(V(t))\abs{V(t)}^{1/2}$. Using the Fourier transform, 
one easily checks that 
$$
\abs{V}^{1/2}R_0(z)^{1/2}\in S_2 \quad
$$
and
$$
\norm{\abs{V}^{1/2} R_0(z)^{1/2}}_{S_2}^2
\leq 
\frac{1}{4\sqrt{\abs{z}}} \int_{-\infty}^\infty \norm{\abs{V(t)}^{1/2}}^2_{S_2}dt
=
\frac{1}{4\sqrt{\abs{z}}}\int_{-\infty}^\infty \norm{V(t)}_{S_1}dt. 
$$
It remains to write 
\begin{multline*}
\norm{ R_0(z)^{1/2}VR_0(z)^{1/2} }_{S_1}
=
\norm{ R_0(z)^{1/2}\abs{V}^{1/2}\sign(V) \abs{V}^{1/2} R_0(z)^{1/2} }_{S_1}
\\
\leq
\norm{\abs{V}^{1/2} R_0(z)^{1/2}}_{S_2}^2,
\end{multline*}
which completes the proof.
\end{proof}

\begin{proof}[Proof of Lemma~\protect\ref{l.b1}]
\mbox{}

1. First we consider the difference of resolvents $R(z)-\wt R(z)$. 
By a well known argument, it suffices to prove that $R(z)-\wt R(z)\in S_1$ 
for al least one value of $z$. 
Using Lemma~\ref{lma.b5}, let us choose $z\in(-\infty,-1)$ with $\abs{z}$ sufficiently large
so that 
\begin{equation}
\norm{R_0(z)^{1/2}B'R_0(z)^{1/2}}\leq 1/2.
\label{b31}
\end{equation}
We have 
\begin{multline*}
H-z=H_0+A^2-B'-z\geq H_0-B'-z
\\
=(H_0-z)^{1/2}(I-R_0(z)^{1/2} B' R_0(z)^{1/2})(H_0-z)^{1/2};
\end{multline*}
the inequality is understood in the sense of the quadratic forms.
Thus we have 
$$
R(z)\leq R_0(z)^{1/2}(I-R_0(z)^{1/2} B' R_0(z)^{1/2})^{-1} R_0(z)^{1/2}
\leq 2 R_0(z).
$$
It follows that the resolvent $R(z)$ can be represented as
\begin{equation}
R(z)=M(z) R_0(z)^{1/2}=R_0(z)^{1/2}M(z)^*
\quad \text{ with }\quad 
\norm{M(z)}^2\leq 2. 
\label{b32}
\end{equation}
The same argument works for $\wt R(z)$, so we have
\begin{equation}
\wt R(z)=\wt M(z) R_0^{1/2}(z)= R_0(z)^{1/2}\wt M(z)^*
\quad \text{ with }\quad 
\norm{\wt M(z)}^2\leq 2. 
\label{b33}
\end{equation}
Now we can rewrite the resolvent identity as 
$$
R(z)-\wt R(z)= 2 R(z) B' \wt R(z)
=
2M(z)(R_0(z)^{1/2}B' R_0(z)^{1/2}) \wt M(z)^*,
$$
and, by Lemma~\ref{lma.b5}, the r.h.s. belongs to the trace class.

2. Consider the difference $g_z(A^+)-g_z(A^-)$. 
We follow the well known argument (see e.g. \cite{BYa}). 
Let $\wt g_z$ be the Fourier transform of the derivative of $g_z$; 
then $\wt g_z$ is absolutely integrable and we can write 
\begin{equation}
g_z(\lambda)=g_z(0)+i
\int_{-\infty}^\infty 
\frac{e^{-i\lambda t}-1}{t} \wt g_z(t) dt
\quad \text{ with } \quad \int_{-\infty}^\infty \abs{\wt g_z(t)}dt<\infty.
\label{b34}
\end{equation}
Thus, we have the representation
\begin{equation}
g_z(A^+)-g_z(A^-)
=
\int_{-\infty}^\infty 
dt\,\, \wt g_z(t) t^{-1} \int_0^s ds e^{-i(t-s)A^+}(A^+-A^-)e^{-isA^-}.
\label{b34a}
\end{equation}
Since 
$$
\norm{e^{-i(t-s)A^+}(A^+-A^-)e^{-isA^-}}_{S_1}\leq \norm{(A^+-A^-)}_{S_1},
$$
the integral in the r.h.s of \eqref{b34a} converges in trace norm. 
\end{proof}

\subsection{Proof of Theorem~\ref{thm.a1}}
1. Denote for brevity $\xi(\lambda)=\xi(\lambda;\wt H,H)$ and 
$\eta(\lambda)=\xi(\lambda;A^+,A^-)$.
According to the Lifshits-Krein trace formula \eqref{a6}, 
the identity  \eqref{a4} can be rewritten as 
$$
-\int_0^\infty \frac{\xi(t)}{(t-z)^2}dt
=
\frac1{2z}\int_{-\infty}^\infty \eta(t)
\left(\frac{\partial}{\partial t}g_z(t)  \right)dt,
$$
which can be rewritten as 
$$
\int_0^\infty \xi(t) \left(\frac{\partial}{\partial z} (t-z)^{-1}\right)dt
=
\int_{-\infty}^\infty\eta(t) \left( \frac{\partial }{\partial z}\frac{1}{\sqrt{t^2-z}}\right) dt.
$$
Integrating over $z$, we get 
\begin{equation}
\int_0^\infty \xi(t)\left( \frac1{t-z}-\frac1{t+1}\right) dt
=
 \int_{-\infty}^\infty \eta(t) \left(\frac1{\sqrt{t^2-z}}-\frac1{\sqrt{t^2+1}} \right) dt.
\label{b15}
\end{equation}

2. Now we would like to take the imaginary parts of both sides of \eqref{b15} and
pass to the limit as $z\to\lambda+i0$ for a.e. $\lambda>0$.
By the well known properties of the Cauchy integrals, we have for a.e. $\lambda>0$
$$
\xi(\lambda)=\lim_{\e\to+0}\frac1\pi  \int_0^\infty \xi(t) \, \Im \frac{1}{t-\lambda-i\e} dt.
$$
Now consider the r.h.s. of \eqref{b15}.
As $\eta\in L^1(\R)$, it is easy to see that the integral
$$
\int_{-\infty}^\infty \frac{\abs{\eta(t)}}{\abs{t^2-\lambda}^{1/2}} dt
$$
converges for a.e. $\lambda>0$.
By the dominated convergence theorem, this ensures that 
$$
\Im \int_{-\infty}^\infty \frac{\eta(t)}{\sqrt{t^2-\lambda-i\e}}dt
\to 
\int_{-\sqrt{\lambda}}^{\sqrt{\lambda}}\frac{\eta(\lambda)}{\sqrt{\lambda-t^2}}dt,
\quad \e\to +0
$$
for a.e. $\lambda>0$. This allows us to pass to the limit in \eqref{b15},
which yields the required result.
\qed

\subsection{Proof of Corollary~\ref{cr.a3}}

1.
In order to check that $D_A$ is Fredholm, we use the following necessary and sufficient condition from 
\cite[Theorem~A.4]{AvronSeilerSimon}: 
\textit{ A closed operator $T$ is Fredholm if and only if $0\notin\sigma_{ess}(T^*T)$ and $0\notin\sigma_{ess}(TT^*)$
and then $\iindex(T)=\dim\Ker(T^*T)-\dim\Ker (T T^*)$. }

Using our assumptions \eqref{a3} and  $0\in\rho(A^-)\cap\rho(A^+)$, 
we can find $a>0$ such that for all sufficiently large $\abs{t}$, one has
$A(t)^2\geq aI$. 
One has
$$
H=H_0+A(t)^2-B'(t)\geq H_0+aI + V(t),
$$
where 
$$
V(t)=-B'(t)+(A(t)^2-aI)E_{A(t)^2}([0,a]).
$$
The operator $(A(t)^2-aI)E_{A(t)^2}([0,a])$ is of a finite rank for all $t\in\R$ and vanishes
for all sufficiently large $\abs{t}$. Thus, 
$V$ satisfies \eqref{b27} and therefore, by Lemma~\ref{lma.b5}, 
$V$ is a relatively form compact perturbation of $H_0$.
It follows that $\sigma_{ess}(H_0+aI+V)=\sigma_{ess}(H_0+aI)=[a,\infty)$. 
Thus, $\inf\sigma_{ess}(H)\geq a>0$. The same argument applies to $\wt H$. 
By the necessary and sufficient condition quoted above, 
$D_A$ is a Fredholm operator and 
\begin{equation}
\iindex D_A=\dim \Ker H-\dim \Ker \wt H.
\label{b29}
\end{equation}

2. 
By the previous step, the spectra of $H$ and $\wt H$ are discrete on $[0,a)$. 
It is well known (see e.g. \cite[Lemma~A.3]{AvronSeilerSimon}) that, 
since $H=D_A^*D_A$ and $\wt H=D_A D_A^*$, one has 
$\sigma(H)\setminus\{0\}=\sigma(\wt H)\setminus\{0\}$ and the 
multiplicities of the eigenvalues coincide. 
Thus, from \eqref{a7} we get
\begin{equation}
\xi(\lambda;\wt H,H)=\dim \Ker H-\dim \Ker \wt H, 
\qquad
\lambda\in(0,a).
\label{b30}
\end{equation}
This identity and its relation to index of $D_A$ is due to \cite{BGS}.

3. 
Combining  Theorem~\ref{thm.a1} with \eqref{b29} and \eqref{b30}, we obtain
$$
\iindex D_A=\frac1\pi\int_{-\sqrt{\lambda}}^{\sqrt{\lambda}}\frac{\xi(t;A^+,A^-)}{\sqrt{\lambda-t^2}}dt, 
\quad \lambda\in(0,a).
$$
Since $0\in\rho(A^-)\cap\rho(A^+)$, 
the function  $\xi(t;A^+,A^-)$ is constant near $t=0$.  
Taking into account the identity 
$$
\frac1\pi \int_{-\sqrt{\lambda}}^{\sqrt{\lambda}}\frac{1}{\sqrt{\lambda-t^2}}dt=1,
$$
we arrive at \eqref{a5}.
\qed

\subsection{Proof of Proposition~\ref{p1} for ``nice'' $A(t)$}

Our proof of Proposition~\ref{p1} consists of two steps: we first 
prove the identity \eqref{a4} for very ``nice" functions $A(t)$ 
and then use approximation argument. 
Here we present the first step; the approximation argument, which has a more technical 
nature, is given in section~\ref{sec.c}.

\begin{lemma}\label{lma.b1}
Suppose that $\dim \H<\infty$. Assume also that $A(t)=A^\pm$ for 
all sufficiently large $\pm t>0$.
Then Proposition~\ref{p1} holds true.
\end{lemma}
\begin{proof}
1.
Of course, in the finite dimensional case all operators belong to the trace class, 
so we only need to prove the identity \eqref{a4}.
In the case $\dim \H=1$, this formula has been proven in \cite[Example~4.1]{BGS}.
Below is a direct generalisation of the argument of \cite{BGS}
 to the matrix-valued case. We call elements of $\H$ ``matrices''.

By analyticity in $z$ it suffices to consider the case of real 
negative $z$. In what follows, we assume $z\in (-\infty,0)$ and denote
$$
\varkappa_\pm=\sqrt{(A^\pm)^2-z}.
$$
The matrices $\k_\pm$ in $\H$ are positive definite.

2. We will compute the trace in the l.h.s. of \eqref{a4} by 
constructing the integral kernels of the resolvents of $H$ and $\wt H$
and evaluating integrals of these kernels over the diagonal. 
Our first aim is to construct the resolvent of $H$ in terms of 
the solutions to the matrix Schrodinger equation.
Here we follow \cite{MOlmedilla}.

Consider the matrix valued solutions $F_\pm(t)$,  $t\in\R$,
to the equation (see \eqref{b1}, \eqref{b1a})
\begin{equation}
-F''_\pm+QF_\pm=zF_\pm,
\label{b2}
\end{equation}
satisfying the asymptotic conditions
\begin{equation}
F_\pm(t)=e^{\mp \k_\pm t},
\quad \pm t>0 \text{ large.}
\label{b4}
\end{equation}
Existence of solutions $F_\pm$ can be proven 
in the usual way by converting \eqref{b2}, \eqref{b4}
into Volterra type integral equations.

We have the relations
\begin{equation}
\begin{split}
F_+(t)&=e^{-\k_- t}a + e^{\k_- t} b,
\quad -t>0\text{ large,}
\\
F_-(t)&=e^{\k_+ t}c + e^{-\k_+ t} d,
\quad t>0\text{ large,}
\end{split}
\label{b5}
\end{equation}
for some matrices $a$, $b$, $c$, $d$ (which, of course, depend 
on $z<0$).
From here it is straightforward to see that if either $a$ or $c$ has
a non-trivial kernel, then $z$ is an eigenvalue of $H$.
But, by definition, $H$ cannot have negative eigenvalues, so we have
\begin{equation}
\Ker a=\Ker c=\{0\}
\label{b6}
\end{equation}
for all $z<0$.

3. For any two solutions $F$, $G$ to the  
equation \eqref{b2}, let us define the Wronskian
$$
W(F,G)=F(t)^*G'(t)-F'(t)^* G(t).
$$
By a direct calculation, the Wronskian does not depend 
on $t$. In particular, using the limiting forms \eqref{b5} 
of the solutions $F_\pm$, we obtain
\begin{equation}
W(F_+,F_-)=2\k_+c=2a^*\k_-.
\label{b7}
\end{equation}
By \eqref{b6}, it follows that $W(F_+,F_-)$ is invertible
(it is here that we need $\dim \H<\infty$).
Now we can construct the integral kernel $R(t,s)$ 
of the resolvent $(H-z)^{-1}$ as in \cite{MOlmedilla}.
We have 
\begin{equation}
R(t,s)=
\begin{cases}
F_+(t)(W(F_+,F_-)^*)^{-1}F_-^*(s),\quad t\geq s,
\\
F_-(t)(W(F_+,F_-))^{-1}F_+^*(s),\quad t<s.
\end{cases}
\label{b7a}
\end{equation}

4. As above, we can construct the integral kernel $\wt R(x,y)$
of the resolvent $(\wt H-z)^{-1}$ in terms of the solutions 
$\wt F_\pm(t)$ to 
\begin{equation}
 -\wt F''_\pm+\wt Q \wt F_\pm=z\wt F_\pm, 
\text{ and } \wt F_\pm(t)=e^{\mp \k_\pm t},
\quad \pm t>0 \text{ large.}
\label{b8}
\end{equation}
By a direct calculation, the function 
$$
\wt F_\pm(t)=(F_\pm'(t)+A(t)F_\pm(t))(\mp\k_\pm+A^\pm)^{-1}
$$
satisfies \eqref{b8}. Also by a direct calculation,
$$
W(\wt F_+,\wt F_-)
=
z(-\k_++A^+)^{-1}W(F_+,F_-)(\k_-+A^-)^{-1}.
$$
This allows us to compute the kernel $\wt R(t,s)$ in terms of the solutions
$F_\pm$:
\begin{equation}
R(t,s)=
\begin{cases}
\frac1z (F'_+(t)+A(t)F_+(t)) (W(F_+,F_-)^*)^{-1}(F'_-(s) +A(s)F_-(s))^*,\quad t\geq s,
\\
\frac1z (F'_-(t)+A(t)F_-(t)) (W(F_+,F_-))^{-1}(F'_+(s)+A(s)F_+(s))^*,\quad t<s.
\end{cases}
\label{b8a}
\end{equation}

5. Now we are ready to compute the trace in the l.h.s. of \eqref{a4}:
$$
\Tr((\wt H-z)^{-1}-(H-z)^{-1})
=
\lim_{R\to\infty}\int_{-R}^R \tr(\wt R(t,t)-R(t,t))dt.
$$
Using our formulas \eqref{b7a}, \eqref{b8a} 
for the kernels $R(t,t)$ and $\wt R(t,t)$, we obtain
$$
\tr(\wt R(t,t)-R(t,t))
=
\tr (W(F_+,F_-)^*)^{-1}(\frac1z((F_-')^*+F_-^*A)(F_+'+AF_+)-F^*_-F_+).
$$
Integrating by parts, after a little algebra we get 
\begin{equation}
\int_{-R}^R\tr(\wt R(t,t)-R(t,t))dt
=\frac1z\tr((W(F_+,F_-)^*)^{-1}F_-^*(F_+'+AF_+))\mid_{-R}^R.
\label{b9}
\end{equation}
Now we can calculate the r.h.s. of \eqref{b9} for large $R>0$,
using formula \eqref{b7}  and the asymptotic forms \eqref{b4}, \eqref{b5} of $F_\pm$: 
\begin{align*}
\tr(W(F_+,F_-)^{-1}F_-^*F_+')\mid_{-R}^R
&=
-\frac12\tr((c^*)^{-1}d^* e^{-2\k_+ R})-\frac12\tr(b a^{-1} e^{-2\k_- R}),
\\
\tr(W(F_+,F_-)^{-1}F_-^*AF_+)\mid_{-R}^R
&=
\frac12\tr(\k_+^{-1} A^+-\k_-^{-1}A^-)
\\
+
\frac12\tr(\k_+^{-1}(c^*)^{-1}d^*  &A^+ e^{-2\k_+ R})
-
\frac12\tr(b a^{-1}\k_-^{-1}A^- e^{-2\k_- R}).
\end{align*}
As $\k_\pm$ are positive definite matrices, we have 
$\norm{e^{-2\k_\pm R}}\to 0$ as $R\to\infty$.
Thus, 
$$
\Tr((\wt H-z)^{-1}-(H-z)^{-1})=\frac1{2z}\tr(\k_+^{-1}A^+-\k_-^{-1}A^-),
$$
as required.
\end{proof}

\section{Approximation argument}\label{sec.c}

First we give a general statement about convergence in both sides of the identity \eqref{a4}
and then construct the approximating sequence $A_n(t)$. 

\subsection{Convergence in \eqref{a4}}
Let $A(t)=A^-+B(t)$ and $A_n(t)=A_n^-+B_n(t)$
be operator families satisfying \eqref{a13}. 
As above, we assume $B(-\infty)=B_n(-\infty)=0$ 
and define $A^+_n=A_n^-+B_n(+\infty)$. 
We also define $H_n=D_{A_n}^*D_{A_n}$,
$\wt H_n=D_{A_n}D^*_{A_n}$, and 
$$
R_n(z)=(H_n-z)^{-1}, \quad \wt R_n(z)=(\wt H_n-z)^{-1}.
$$
\begin{lemma}\label{l.c1}
Assume that $\Dom A^-\subset \Dom A_n^-$ for all $n$ and 
$A_n^- f\to A^-f$ for all $f\in\Dom(A^-)$. Next, assume that 
\begin{equation}
\int_{-\infty}^\infty \norm{B'_n(t)-B'(t)}_{S_1}dt\to 0 \quad \text{ as $n\to\infty$.}
\label{c5}
\end{equation}
Then, for all $z\in(-\infty, -1)$ with sufficiently large $\abs{z}$, one has
\begin{gather}
\norm{(g_z(A^+)-g_z(A^-))-(g_z(A_n^+)-g_z(A_n^-))}_{S_1}\to 0,
\label{c6}
\\
\norm{(R(z)-\wt R(z))-(R_n(z)-\wt R_n(z))}_{S_1}\to 0.
\label{c7}
\end{gather}
\end{lemma}

\begin{remark}\label{rmk.c1}
1. Assumption \eqref{c5} implies that 
\begin{equation}
\norm{(A_n^+-A_n^-)-(A^+-A^-)}_{S_1}\to 0.
\label{c8}
\end{equation}

2. Our assumptions on $A^-$, $A_n^-$ imply that $A_n^-\to A^-$ in strong resolvent 
sense; see \cite[Theorem VIII.25(a)]{ReedSimon1} and its proof. 
Combining this with \eqref{c8}, we see that also $A_n^+\to A^+$ in strong resolvent 
sense. 
\end{remark}

We will repeatedly make use of the following well known fact,
which holds true for any Schatten-von Neumann class $S_p$, $p\geq 1$,
although we will only need it for the case $p=1,2$:
\begin{proposition}\label{prp.b1}
Let $T_n$ be a sequence of bounded operators in a Hilbert space which
converges strongly to zero and let $M\in S_p$; then $\norm{T_n M}_{S_p}\to0$.
If $T_n^*$ also converges strongly to zero, then $\norm{MT_n }_{S_p}\to0$.
\end{proposition}
The first part of this Proposition can be found, for example, in \cite[Lemma 6.1.3]{Yafaev},
and the second part follows immediately by conjugation, since
$\norm{MT_n}_{S_p}=\norm{T_n^*M^*}_{S_p}$.

\begin{proof}[Proof of \eqref{c6}]
Writing the representation \eqref{b34a}, we get
\begin{gather}
(g_z(A^+)-g_z(A^-))-(g_z(A^+_n)-g_z(A^-_n))
=
\int_{-\infty}^\infty dt\,\,  \wt g_z(t) t^{-1} \int_0^t ds\,\, K_n(t,s),
\label{c9}
\\
K_n(t,s)
=e^{-i(t-s)A^+}(A^+-A^-)e^{-isA^-} 
-
e^{-i(t-s)A^+_n}(A^+_n-A^-_n)e^{-isA^-_n}.
\notag
\end{gather}
We would like to use the dominated convergence theorem in order to prove that
the r.h.s. of \eqref{c9} converges to zero in the trace norm. 
First note that
$$
\norm{K_n(t,s)}_{S_1}\leq \norm{A^+-A^-}_{S_1}+\norm{A_n^+-A_n^-}_{S_1}
$$
and, by \eqref{c8}, the r.h.s. is bounded uniformly in $n$ by some constant $C$. 
This gives  an integrable bound for the integrand in the r.h.s. of
\eqref{c9}. 

Next, we claim that $\norm{K_n(t,s)}_{S_1}\to 0$ for all $t,s$. Indeed, 
we can write 
\begin{multline}
K_n(t,s)=(e^{-i(t-s)A^+}-e^{-i(t-s)A^+_n})(A^+-A^-)e^{-isA^-} 
\\
+e^{-i(t-s)A^+_n}(A^+-A^-)(e^{-isA^-} -e^{-isA^-_n})
+e^{-i(t-s)A^+_n}(A^+-A^--A^+_n+A^-_n)e^{-isA^-_n}.
\label{c10}
\end{multline}
The last term in the r.h.s. converges to zero by \eqref{c8}. 
Next, since $A_n^\pm\to A^\pm$ in strong resolvent sense (see Remark~\ref{rmk.c1}), 
by \cite[Theorem~VIII.21]{ReedSimon1} we have strong convergence 
$e^{itA_n^\pm}\to e^{itA^\pm}$ as $n\to\infty$. 
Thus, by Proposition~\ref{prp.b1}, the first two terms in the r.h.s. of \eqref{c10} 
converge to zero in the trace norm. By dominated convergence, this proves \eqref{c6}.
\end{proof}

The proof of \eqref{c7} requires a little more work. 
First we need an abstract lemma which ensures the strong convergence of resolvents. 
I am indebted to Nikolai Filonov for providing the proof of this lemma. 

Let $D$  be a closed densely defined operator in a Hilbert space such that $D^*$ is also densely 
defined. 
For each $n$, let $D_n$ be a closed densely defined operator such that $D_n^*$ is also 
densely defined and $\Dom D\subset \Dom D_n$ and $\Dom D^*\subset \Dom D_n^*$.
\begin{lemma}\label{lma.c1}
Assume the above conditions and assume that for all $f\in\Dom D$ one has
$\norm{D_nf-Df}\to0$ and for all $f\in\Dom D^*$ one has $\norm{D_n^*f-D^*f}\to0$
as $n\to \infty$.
Then for all $z\in\C\setminus[0,\infty)$, 
one has the strong convergence
\begin{equation}
(D_n^*D_n-z)^{-1}\to (D^*D-z)^{-1},
\quad n\to \infty.
\label{c1}
\end{equation}
\end{lemma}
\begin{proof}
1. 
By \cite[Chapter VIII, Problem 20]{ReedSimon1}, 
it suffices to prove the weak convergence in \eqref{c1}.

2.
Fix  $z\in\C\setminus[0,\infty)$, 
$f\in \H$ and denote $\f_n=(D_n^*D_n-z)^{-1}f$. We need to prove
that $\f_n\to \f$ weakly, where $\f=(D^*D-z)^{-1}f$. 
First note that 
\begin{equation}
\norm{\varphi_n}\leq \frac{\norm{f}}{\dist(z,[0,\infty))}.
\label{c1a}
\end{equation}
Next, we have 
$$
\norm{D_n \f_n}^2- z \norm{\f_n}^2
=
(f,\f_n)
=(f,(D_n^*D_n-z)^{-1}f), 
$$
and so 
\begin{equation}
\norm{D_n \varphi_n}^2\leq \abs{z}\norm{\varphi_n}^2+\abs{(f,(D_n^*D_n-z)^{-1}f)}\leq C(z)\norm{f}^2.
\label{c2}
\end{equation}
By \eqref{c1a}, \eqref{c2} and  the weak compactness of the unit ball in a Hilbert space, from 
the sequence $\f_n$ one can choose a subsequence $\f_{n_k}$
such that $ \f_{n_k}\to \tilde \f$ and $D_{n_k} \f_{n_k}\to \psi$ 
weakly for some elements $\tilde \f$, $\psi$ in $\H$. 

3.
Let us prove that $\tilde \f\in \Dom D$ and $\psi=D \tilde \f$. 
For any $\chi\in \Dom D^*$, we have 
$$
(D_{n_k} \f_{n_k},\chi)\to(\psi,\chi),
$$
and so 
$$
( \f_{n_k},D_{n_k}^*\chi)\to(\psi,\chi).
$$
Since $\norm{D_{n_k}^*\chi-D^*\chi}\to0$ by our assumptions, we get 
\begin{equation}
(\wt \f, D^*\chi)=(\psi, \chi)
\label{c3}
\end{equation}
for all $\chi\in \Dom D^*$;
it follows that $\tilde \f\in \Dom D$ and $\psi=D \tilde \f$. 

4.
Next, we have for any $\chi\in\Dom D$:
\begin{equation}
(D_n\f_n, D_n\chi)-z(\f_n,\chi)=(f,\chi).
\label{c4}
\end{equation}
By the previous step, $\f_{n_k}\to \tilde \f$ and $D_{n_k} \f_{n_k}\to \psi$
weakly, and by the hypothesis, $D_n\chi\to D\chi$ strongly. 
Passing to the limit in \eqref{c4} over the subsequence $n_k$, 
we get
$$
(\psi, D\chi)-z(\tilde\f,\chi)=(f,\chi)
$$
for all $\chi\in \Dom D$. 
It follows that $\psi\in \Dom D^*$ and $D^*\psi-z\tilde\f=f$. 
Recalling that $\psi=D\tilde \f$, we get that $\tilde \f\in \Dom (D^*D)$ and 
$(D^*D-z)\tilde \f=f$. 
Thus, we have proven that $\tilde \f=\f$.

5.
We have proven weak convergence $\f_n\to \f$
over a subsequence $n_k$. But  we could have started 
from an arbitrary subsequence of $\f_n$ and proven that it has 
a subsubsequence which weakly converges to $\f$. 
This proves that actually the whole sequence $\f_n$ weakly converges to $\f$. 
\end{proof}

\begin{proof}[Proof of \eqref{c7}]

1. By  Lemma~\ref{lma.b5} and the uniform boundedness of the integrals 
$\int \norm{B_n'(t)}_{S_1}dt$, one can choose $a<-1$ such that for all 
$z\in(-\infty,a)$, the estimates
$$
\norm{R_0(z)^{1/2} B' R_0(z)^{1/2}}\leq 1/2,\quad
\sup_n \norm{R_0(z)^{1/2} B_n' R_0(z)^{1/2}}\leq 1/2
$$
hold true. Then, as in the proof of Lemma~\ref{l.b1},
we get \eqref{b32}, \eqref{b33}, and also 
\begin{align*}
R_n(z)&=M_n(z)R_0(z)^{1/2} =R_0(z)^{1/2}M_n(z)^*,
\\
\wt R_n(z)&=\wt M_n(z) R_0(z)^{1/2} =R_0(z)^{1/2}\wt M_n(z)^*
\end{align*}
with $\norm{M_n(z)}^2\leq 2$ and $\norm{\wt M_n(z)}^2\leq 2$. 
In what follows, we fix $z\in(-\infty,a)$ as above and suppress the 
dependance of $z$ in our notation for brevity.

2. Note that by Lemma~\ref{lma.c1}, we have the strong convergence 
of resolvents $R_n\to R$, $\wt R_n\to \wt R$. Moreover, we claim that 
for the operators $M_n$, $\wt M_n$ we have the strong convergence
$M_n\to M$, 
$\wt M_n\to \wt M$. 
Indeed, for all $f\in\Dom(H_0-z)^{1/2}$ one has 
$$
M_n f=R_n(H_0-z)^{1/2}f\to R(H_0-z)^{1/2}f=Mf.
$$
Since the norms of $M_n$ are uniformly bounded, we get the strong 
convergence $M_n\to M$. 
In the same way, we obtain the strong convergence $\wt M_n\to \wt M$.

3. Using the resolvent identity, we obtain
\begin{multline}
(R_n-\wt R_n)-(R-\wt R)
=
R_n  B_n' \wt R_n-R B' \wt R
\\
=
R_n( B_n' -B')\wt R_n+(R_n-R)B'\wt R_n
+RB'(\wt R_n-\wt R).
\label{b21}
\end{multline}
Let us consider separately each of the three terms in the r.h.s. of \eqref{b21}.

4. For the first term, we have 
$$
R_n(B_n' -B')\wt R_n
=
M_n(R_0^{1/2}(B_n'-B')R_0^{1/2})\wt M_n^*
$$
and the r.h.s. converges to zero in the trace norm by Lemma~\ref{lma.b5} 
and assumption \eqref{c5}.

5. Consider the second term  in the r.h.s. of \eqref{b21}. We have:
$$
(R_n-R)B'\wt R_n
=
(M_n-M)(R_0^{1/2}B'R_0^{1/2})\wt M_n^*.
$$
Since $R_0^{1/2}B'R_0^{1/2}$ is a trace class operator, and  
$M_n\to M$ strongly as $n\to\infty$,  
by Proposition~\ref{prp.b1}, we obtain that the r.h.s. converges to zero in the trace norm.

6. Finally, the third term in the r.h.s. of \eqref{b21} can be considered
similarly to the second one:
$$
[RB'(\wt R_n-\wt R)]^*=(\wt R_n-\wt R)B'R=(\wt M_n-\wt M)(R_0^{1/2}B'R_0^{1/2})M^*,
$$
and the r.h.s. goes to zero in the trace norm as $n\to\infty$. 
\end{proof}

\subsection{Constructing the approximating family $A_n(t)$}
We will approximate $A(t)$ in two steps. First, we approximate an arbitrary 
finite rank family $A(t)$ by the ones with compactly supported $A'(t)$.
Next, we approximate an arbitrary family by finite rank families. 

\begin{lemma}\label{lma.c2}
Let $\dim\H<\infty$. Then Proposition~\ref{p1} holds true. 
\end{lemma}
\begin{proof}
By analyticity in $z$, it suffices to prove \eqref{a4} for $z\in(-\infty,-1)$ with 
sufficiently large $\abs{z}$. Then we can use Lemma~\ref{l.c1}.

For a given family $A(t)$, let us construct a sequence of families $A_n(t)$ 
such that for each $n$ and all sufficiently large $\pm t>0$, we have
$A_n(t)=A^\pm$. This is not difficult to do. 
Indeed, let $A_n(t)$ be such that $A_n(t)=A(t)$ for $t\in[-n,n]$, 
$A_n(t)=A^-$ for $t\leq -n-1$, $A_n(t)=A^+$ for $t\geq n+1$, 
and $A_n(t)$ is obtained by linear interpolation on $[-n-1,-n]$
and $[n,n+1]$. Explicitly, 
\begin{gather*}
A_n(t)=A^-+(t+n+1)(A(-n)-A^-), \quad t\in[-n-1,-n],
\\
A_n(t)=(t-n)A^++(n+1-t)A(n), \quad t\in [n,n+1].
\end{gather*}
Then we have 
\begin{multline*}
\int_{-\infty}^\infty \norm{A_n'(t)-A'(t)}_{S_1}dt
\leq
\int_n^\infty (\norm{A_n'(t)}_{S_1}+\norm{A'(t)}_{S_1})dt
\\
+
\int_{-\infty}^{-n} (\norm{A_n'(t)}_{S_1}+\norm{A'(t)}_{S_1})dt
\to 0
\end{multline*}
as $n\to\infty$. 
By Lemma~\ref{lma.b1}, the identity \eqref{a4} holds true 
for the families $A_n(t)$. By Lemma~\ref{l.c1}, we can pass to the limit
as $n\to\infty$ in both sides of \eqref{a4}, which yields the required result. 
\end{proof}

Next, we approximate an arbitrary family $A(t)$ by finite rank families.

\begin{lemma}\label{lma.b4}
There exists a sequence of finite rank orthogonal projections $P_n$ in $\H$
such that:

(i) $P_n\to I$ strongly as $n\to\infty$;
 
(ii) $\Ran P_n\subset \Dom A^-$ for all $n$;

(iii) for all $f\in \Dom (A^-)$, one has 
$\norm{P_nA^- P_nf-A^-f}\to 0$ as $n\to\infty$.
\end{lemma}
\begin{proof}
For any $k\in\Z$,
let $E_k$ be the spectral projection of the operator $A^-$ associated with the
interval $(k,k+1]$, and let $Q_{j}^{(k)}$ be a sequence of finite 
rank projections such that 
$\Ran Q_{j}^{(k)}\subset \Ran E_k$ and $Q_{j}^{(k)}\to E_k$ 
strongly as $j\to\infty$.
Take $P_n=\sum_{k=-n}^n Q_{n}^{(k)}$. 
Then (i), (ii) are obvious. 
Let us prove (iii). If $f\in \Dom A^-$, then 
$f=\sum_{k=-\infty}^\infty f_k$, where $f_k=E_k f$ and 
$\sum_{k=-\infty}^\infty (k^2+1)\norm{f_k}^2<\infty$. 
Given such $f$ and $\e>0$, we can choose $N$ 
sufficiently large so that
$\sum_{\abs{k}\geq N} (k^2+1)\norm{f_k}^2<\e^2$;
denote $g_1= \sum_{\abs{k}< N}f_k$ and 
$g_2= \sum_{\abs{k}\geq N}f_k$.
Then $\norm{P_n A^- P_n g_2}\leq \e$ and  $\norm{A^-g_2}\leq \e$.
On the other hand, it is easy to see that 
$\norm{P_nA^- P_ng_1-A^-g_1}\to 0$ as $n\to\infty$.
This proves (iii). 
\end{proof}

\begin{proof}[Proof of Proposition~\ref{p1}]
1. Let $A_n(t)=P_n A(t) P_n$, $B_n(t)=P_n B(t) P_n$, 
where $P_n$ are as constructed in Lemma~\ref{lma.b4}. 
We claim that 
$$
\int_{-\infty}^\infty \norm{B_n'(t)-B'(t)}_{S_1}dt\to 0
\quad \text{as $n\to\infty$. }
$$
Indeed, let us apply the dominated convergence theorem. 
First, we have 
$$
\norm{B_n'(t)-B'(t)}_{S_1}
\leq 
\norm{B_n'(t)}_{S_1}+\norm{B'(t)}_{S_1}
\leq 
2\norm{B'(t)}_{S_1}.
$$
Next, for all $t\in \R$, we have
$$
B_n'(t)-B'(t)=(P_n-I)B'(t)P_n+B'(t)(P_n-I),
$$
and the right hand side converges to zero by Proposition~\ref{prp.b1},
since $P_n\to I$ and $B'(t)\in S_1$. 

2. By Lemma~\ref{lma.c2}, the identity \eqref{a4} holds true for the families
$A_n(t)$. By Lemma~\ref{l.c1}, we can pass to the limit as $n\to\infty$ in both 
sides of \eqref{a4} when $z\in(-\infty,-1)$, $\abs{z}$ large. 
By analyticity, this yields the required result for all $z$. 
\end{proof}

\end{document}